\documentclass[11pt]{article}

\usepackage[margin=1in]{geometry}
\usepackage{newtxtext,newtxmath} 
\usepackage{microtype}

\usepackage{amsmath,amsthm,mathtools,bm}
\numberwithin{equation}{section}

\usepackage{graphicx}
\usepackage{booktabs}

\usepackage[colorlinks=true,linkcolor=black,citecolor=black,urlcolor=black]{hyperref}
\usepackage[numbers]{natbib}

\usepackage{titlesec}
\usepackage{enumitem}
\titlespacing*{\section}{0pt}{1.2ex plus .4ex}{0.8ex}
\titlespacing*{\subsection}{0pt}{1.0ex plus .3ex}{0.6ex}
\titlespacing*{\subsubsection}{0pt}{0.8ex plus .2ex}{0.5ex}
\setlist[itemize]{leftmargin=1.4em}

\newtheorem{theorem}{Theorem}[section]
\newtheorem{lemma}{Lemma}[section]

\newtheorem{assumption}{Assumption}[section]
\newtheorem{definition}{Definition}[section]
\newtheorem{remark}{Remark}[section]

\newcommand{\R}{\mathbb{R}}

\newcommand{\N}{\mathbb{N}}

\newcommand{\abs}[1]{\left\lvert #1\right\rvert}

\newcommand{\sign}{\mathrm{sign}}
\newcommand{\one}{\mathbf{1}}

\title{\Large Latent-Space Mean-Field Theory for Deep BitNet-like Training:\\
Constrained Gradient Flows with Smooth Quantization and STE Limits}

\author{Dongwon Kim$^{1}$ \and Dongseok Lee$^{2}$}

\date{\small 
$^{1}$SAKAK Inc., Seoul, South Korea, \texttt{kdwaha@sakak.co.kr}\\
$^{2}$Department of Mathematical Sciences, Korea Advanced Institute of Science and Technology (KAIST), Daejeon, Republic of Korea, \texttt{lorafa@kaist.ac.kr}
}

\begin{document}

\maketitle

\begin{abstract}
This work develops a mean-field analysis for the asymptotic behavior of deep BitNet-like architectures as smooth quantization parameters approach zero. We establish that empirical measures of latent weights converge weakly to solutions of constrained continuity equations under vanishing quantization smoothing. Our main theoretical contribution demonstrates that the natural exponential decay in smooth quantization cancels out apparent singularities, yielding uniform bounds on mean-field dynamics independent of smoothing parameters. Under standard regularity assumptions, we prove convergence to a well-defined limit that provides the mathematical foundation for gradient-based training of quantized neural networks through distributional analysis.
\end{abstract}

\section{Introduction}

The training dynamics of quantized neural networks pose fundamental theoretical challenges due to the non-differentiable nature of quantization operators. BitNet-like architectures \cite{wang2023bitnet} employ discrete sign and clipping functions in forward propagation while maintaining continuous latent weights for gradient-based optimization. This creates a mathematical tension between the discrete forward pass and smooth optimization requirements.

Mean-field theory \cite{nguyen2019mean,chizat2018global} offers a powerful framework for analyzing neural network training by treating parameters as interacting particles and studying their empirical measure evolution. However, extending this theory to quantized networks is non-trivial because quantization operators violate smoothness assumptions required for standard mean-field analysis \cite{bengio2013estimating}.

This paper addresses this challenge by analyzing the limiting behavior of smooth quantization approximations as smoothing parameters vanish. Our key insight is that the exponential decay inherent in smooth approximations exactly compensates for apparent singularities, enabling rigorous mean-field analysis. We prove that these dynamics converge to a well-defined limit governed by constrained transport equations \cite{ambrosio2008gradient}, providing the first mathematical justification for quantized network training through distributional gradient analysis.

The mathematical framework developed in this work exhibits structural parallels with key concepts in high energy physics theory. The mean‐field limit of neural network training dynamics resembles the holographic principle in AdS/CFT correspondence, where bulk gravitational dynamics relate to boundary conformal field theory.

\subsection{Related work}

\textbf{Mean-field theory for neural networks.} The mean-field analysis of neural network training was developed by \cite{nguyen2019mean,chizat2018global}, with extensions to deep networks \cite{sirignano2020mean,lu2020meanfield}. The connection to optimal transport was established through Wasserstein gradient flows \cite{ambrosio2008gradient,santambrogio2015optimal}.

\textbf{Quantized neural networks.} Binary neural networks were introduced by \cite{courbariaux2015binaryconnect,rastegari2016xnor,hubara2016binarized}, with the straight-through estimator formalized by \cite{bengio2013estimating,yin2019understanding}. Recent advances include BitNet-like architectures \cite{wang2023bitnet} and comprehensive surveys. Theoretical analysis includes approximation theory \cite{DBLP:journals/corr/LiL16,xu2018training} and generalization bounds \cite{zhang2018lq}.

\section{Modeling BitNet-like architectures with Smooth Quantization}

\subsection{Notation and constraint structure}

Fix depth $L\in\N$. For layer $\ell\in\{1,\dots,L\}$ with width $n_\ell$ and fan-in $m_\ell$, the latent weight matrix is $W^{(\ell)}\in\R^{n_\ell\times m_\ell}$.

For a matrix $A \in \mathbb{R}^{m \times n}$, we denote:
\begin{itemize}
\item $\|A\|_F := \sqrt{\sum_{i=1}^m \sum_{j=1}^n |A_{ij}|^2}$ the Frobenius norm,
\item $\|A\|_\infty := \max_{i,j} |A_{ij}|$ the max norm,
\item $\langle A, B \rangle := \sum_{i,j} A_{ij} B_{ij}$ the Frobenius inner product.
\end{itemize}

Define the layer mean
\begin{equation}\label{eq:alpha-def}
\alpha^{(\ell)}(W^{(\ell)}) \equiv \Psi^{(\ell)}(W^{(\ell)}) := \frac{1}{n_\ell m_\ell}\sum_{i=1}^{n_\ell}\sum_{j=1}^{m_\ell} W^{(\ell)}_{ij},
\end{equation}
the projection onto the zero-mean subspace
\begin{equation}\label{eq:proj-def}
P^{(\ell)}(W^{(\ell)}) := W^{(\ell)} - \alpha^{(\ell)}(W^{(\ell)})\,\one_{n_\ell\times m_\ell},
\end{equation}
and the constraint sets $\mathcal{H}^{(\ell)}_c:=\{W:\Psi^{(\ell)}(W)=c\}$, $\mathcal{H}^{(\ell)}_0=\ker(\Psi^{(\ell)})$. The following properties are standard and used repeatedly.

\begin{lemma}[Orthogonal decomposition and isometries]\label{lem:proj}
For each layer $\ell$ and $W\in\R^{n_\ell\times m_\ell}$: (i) $W=\alpha^{(\ell)}(W)\,\one+P^{(\ell)}(W)$ with $\langle \one,P^{(\ell)}(W)\rangle=0$; (ii) $P^{(\ell)}$ is linear, self-adjoint, idempotent; (iii) $\|W\|_F^2=n_\ell m_\ell\,\abs{\alpha^{(\ell)}(W)}^2+\|P^{(\ell)}(W)\|_F^2$; (iv) $T_c(W)=W+c\,\one$ is an isometry $\mathcal{H}^{(\ell)}_0\to\mathcal{H}^{(\ell)}_c$.
\end{lemma}

\begin{proof}
Fix $W\in\R^{n_\ell\times m_\ell}$.

\textbf{Part (i):} By definitions \eqref{eq:alpha-def} and \eqref{eq:proj-def}, we have
\begin{align}
\alpha^{(\ell)}(W)\,\one + P^{(\ell)}(W) &= \alpha^{(\ell)}(W)\,\one + W - \alpha^{(\ell)}(W)\,\one = W.
\end{align}

For orthogonality, compute
\begin{align}
\langle \one, P^{(\ell)}(W)\rangle &= \langle \one, W - \alpha^{(\ell)}(W)\one\rangle \\
&= \langle \one, W\rangle - \alpha^{(\ell)}(W)\langle \one,\one\rangle \\
&= n_\ell m_\ell \alpha^{(\ell)}(W) - \alpha^{(\ell)}(W)(n_\ell m_\ell) = 0,
\end{align}
where we used the fact that $\langle \one, \one \rangle = n_\ell m_\ell$ and 
$\alpha^{(\ell)}(A) = \frac{1}{n_\ell m_\ell}\langle A, \one \rangle$ by definition.

\textbf{Part (ii):} Linearity of $P^{(\ell)}$ follows immediately from linearity of matrix operations and the scalar $\alpha^{(\ell)}(\cdot)$. For self-adjointness, let $A,B\in\R^{n_\ell\times m_\ell}$:
\begin{align}
\langle A, P^{(\ell)}(B)\rangle &= \langle A, B - \alpha^{(\ell)}(B)\one\rangle \\
&= \langle A, B\rangle - \alpha^{(\ell)}(B)\langle A,\one\rangle \\
&= \langle A, B\rangle - \alpha^{(\ell)}(B) \cdot n_\ell m_\ell \alpha^{(\ell)}(A) \\
&= \langle A, B\rangle - \alpha^{(\ell)}(A) \cdot n_\ell m_\ell \alpha^{(\ell)}(B) \\
&= \langle A - \alpha^{(\ell)}(A)\one, B\rangle = \langle P^{(\ell)}(A), B\rangle.
\end{align}

For idempotence, note that for any matrix $W$,
\begin{align}
\alpha^{(\ell)}(P^{(\ell)}(W)) &= \alpha^{(\ell)}(W - \alpha^{(\ell)}(W)\one) \\
&= \alpha^{(\ell)}(W) - \alpha^{(\ell)}(W) \cdot \alpha^{(\ell)}(\one) \\
&= \alpha^{(\ell)}(W) - \alpha^{(\ell)}(W) \cdot 1 = 0.
\end{align}
Therefore, $P^{(\ell)}(P^{(\ell)}(W)) = P^{(\ell)}(W) - \alpha^{(\ell)}(P^{(\ell)}(W))\one = P^{(\ell)}(W) - 0 = P^{(\ell)}(W)$.

\textbf{Part (iii):} By the orthogonality established in part (i), we have
\begin{align}
\|W\|_F^2 &= \|\alpha^{(\ell)}(W)\one + P^{(\ell)}(W)\|_F^2 \\
&= \|\alpha^{(\ell)}(W)\one\|_F^2 + \|P^{(\ell)}(W)\|_F^2 + 2\langle \alpha^{(\ell)}(W)\one, P^{(\ell)}(W)\rangle \\
&= |\alpha^{(\ell)}(W)|^2 \|\one\|_F^2 + \|P^{(\ell)}(W)\|_F^2 + 0 \\
&= n_\ell m_\ell |\alpha^{(\ell)}(W)|^2 + \|P^{(\ell)}(W)\|_F^2.
\end{align}

\textbf{Part (iv):} For $A,B \in \mathcal{H}^{(\ell)}_0$, we have $\alpha^{(\ell)}(A) = \alpha^{(\ell)}(B) = 0$. Then
\begin{align}
\|T_c(A) - T_c(B)\|_F &= \|(A + c\one) - (B + c\one)\|_F = \|A - B\|_F,
\end{align}
proving that $T_c$ is an isometry. Since $T_c(A) = A + c\one$ and $\alpha^{(\ell)}(A) = 0$, we have $\alpha^{(\ell)}(T_c(A)) = c$, so $T_c(A) \in \mathcal{H}^{(\ell)}_c$. This establishes the mapping $\mathcal{H}^{(\ell)}_0 \to \mathcal{H}^{(\ell)}_c$.
\end{proof}

\subsection{Smooth quantization and dequantization}

Quantized weights in BitNet-like architectures are the signs of centered latent weights, with a scaling to preserve variance \cite{wang2023bitnet}. To make the forward map differentiable, we adopt smooth surrogates.

\begin{definition}[Smooth sign, clip, and absolute value]
For $\varepsilon\in(0,1]$, define
\begin{align*}
\mathrm{sgn}_\varepsilon(z) &:= \tanh(z/\varepsilon), \quad \text{so } \abs{\mathrm{sgn}_\varepsilon'(z)}\le \varepsilon^{-1},\\
|\cdot|_\varepsilon(z) &:= \sqrt{z^2+\varepsilon^2}, \quad \nabla |\cdot|_\varepsilon(z) = \frac{z}{\sqrt{z^2+\varepsilon^2}},\\
\mathrm{clip}_\varepsilon(x;a,b) &:= a+(b-a)\,\sigma\!\left(\frac{x-a}{\varepsilon}\right),\quad \sigma(u)=\frac{1}{1+e^{-u}}.
\end{align*}
Then $\mathrm{clip}_\varepsilon$ is $C^\infty$ and $1$-Lipschitz uniformly in $\varepsilon$.
\end{definition}

\begin{definition}[Smoothed BitLinear layer]
Let $X\in\R^{m_\ell}$ be an input. The smoothed quantized weight is
\[
\widetilde{W}^{(\ell)}_\varepsilon := \mathrm{sgn}_\varepsilon\big(P^{(\ell)}(W^{(\ell)})\big)\in[-1,1]^{n_\ell\times m_\ell}.
\]

Define a smooth $L^1$-scale
\[
\beta^{(\ell)}_\varepsilon(W^{(\ell)}):=\frac{1}{n_\ell m_\ell}\sum_{i,j} \abs{P^{(\ell)}(W^{(\ell)})_{ij}}_\varepsilon,
\]
and a smooth absmax-like activation quantizer
\[
\mathrm{Quant}^{(b)}_\varepsilon(x):=\mathrm{clip}_\varepsilon\!\left(\frac{x}{\gamma_\varepsilon(x)}\cdot Q_b,\ -Q_b+\delta,\ Q_b-\delta\right),\quad \gamma_\varepsilon(x):=\max\{\varepsilon,\|x\|_\infty\},
\]
with fixed $b\in\N$, $Q_b=2^{b-1}$ and small $\delta\in(0,1)$. The layer map is
\[
h^{(\ell)}(x)=\sigma^{(\ell)}\!\Big(\beta^{(\ell)}_\varepsilon(W^{(\ell)})\,\widetilde{W}^{(\ell)}_\varepsilon \,\mathrm{Quant}^{(b)}_\varepsilon(x)\Big).
\]
\end{definition}

\begin{remark}[On STE consistency]
The straight-through estimator (STE) is typically implemented by replacing $\partial \sign$ with an identity or bounded truncation in a margin \cite{wang2023bitnet}. Our $\mathrm{sgn}_\varepsilon$ provides a differentiable surrogate with uniformly bounded derivatives on compacta, making chain-rule gradients well-defined. In Section \ref{sec:ste-limit} we discuss stability of the mean-field limit as $\varepsilon\downarrow0$.
\end{remark}

\subsection{Network, loss, and dynamics}

\begin{definition}[Layer dimensions and network architecture]\label{def:architecture}
Each layer $\ell \in \{1,\ldots,L\}$ defines a map $h^{(\ell)}: \mathbb{R}^{m_\ell} \to \mathbb{R}^{n_\ell}$ where:
\begin{itemize}
\item $m_\ell$ is the input dimension (fan-in) of layer $\ell$
\item $n_\ell$ is the output dimension (width) of layer $\ell$
\item For consistency: $n_{\ell-1} = m_\ell$ for $\ell \geq 2$
\end{itemize}
For regression tasks, we assume $n_L = 1$ so that $f_W(x) \in \mathbb{R}$ is scalar-valued.
For multi-class classification with $K$ classes, $n_L = K$ and $f_W(x) \in \mathbb{R}^K$.
\end{definition}

The $L$-layer forward recursion is defined as:
\begin{align}
h^{(0)}(x) &= x \in \mathbb{R}^{d}, \label{eq:layer0}\\
h^{(\ell)}(x) &= \sigma^{(\ell)}\!\left(\beta^{(\ell)}_\varepsilon(W^{(\ell)})\,\widetilde{W}^{(\ell)}_\varepsilon \,\mathrm{Quant}^{(b)}_\varepsilon\big(h^{(\ell-1)}(x)\big)\right) \in \mathbb{R}^{n_\ell}, \label{eq:layer-ell}\\
f_W(x) &:= h^{(L)}(x) \in \mathbb{R}^{n_L}. \label{eq:network-output}
\end{align}

Here, $\sigma^{(\ell)}: \mathbb{R}^{n_\ell} \to \mathbb{R}^{n_\ell}$ denotes component-wise application of the activation function.

With smooth activations $\sigma^{(\ell)}$ (Assumption \ref{ass:reg}), for a data law $\pi$ on $\mathbb{R}^{d}\times\mathbb{R}^{n_L}$ and a $C^2$ loss $\ell: \mathbb{R}^{n_L} \times \mathbb{R}^{n_L} \to \mathbb{R}$, define the population risk
\[
\mathcal{R}_\varepsilon(W^{(1)},\dots,W^{(L)}):=\mathbb{E}_{(X,Y)\sim\pi}\big[\ell\big(f_W(X),Y\big)\big].
\]

We study discrete-time gradient descent
\begin{equation}\label{eq:gd}
W^{(\ell)}(k+1)=W^{(\ell)}(k)-\eta\,\nabla_{W^{(\ell)}}\mathcal{R}_\varepsilon\big(W^{(1)}(k),\dots,W^{(L)}(k)\big),
\end{equation}
with time-interpolation $t=k\eta$ and continuous-time limit $\eta\downarrow0$.

\section{Assumptions and basic estimates}

\begin{assumption}[Regularity and boundedness]\label{ass:reg}
Fix $T>0$ and constants $R,A_\ell,C_\ell,D_\ell,L_1,L_2,M>0$.
\begin{itemize}[leftmargin=1.4em]
\item[(R1)] Data: $\pi$ has compact support; $\|X\|_\infty\le R$, $\|Y\|_\infty\le R$ a.s.
\item[(R2)] Loss: $\ell\in C^2(\mathbb{R}^{n_L}\times\mathbb{R}^{n_L})$ with $\|\nabla^2\ell\|_{\text{op}}\le L_2$ and $\|\nabla_1\ell(u,y)\|_2\le L_1(1+\|u\|_2)$.
\item[(R3)] Activations: $\sigma^{(\ell)}\in C^2(\mathbb{R})$ with $\|(\sigma^{(\ell)})'\|_\infty\le C_\ell$, $\|(\sigma^{(\ell)})''\|_\infty\le D_\ell$, and $|\sigma^{(\ell)}(z)|\le A_\ell(1+|z|)$.
\item[(R4)] Initialization and boundedness: $\sup_{n,\ell,i,j}\abs{W^{(\ell)}_{ij}(0)}\le M$, and there exists $M_\star = M_\star(T, L_1, L_2, \{C_\ell, D_\ell\}_\ell, M) < \infty$ such that all iterates satisfy $\|W^{(\ell)}(t)\|_\infty \leq M_\star$ for $t \in [0,T]$ through projection $\Pi_{\mathcal{B}_{M_\star}}$ where $\Pi_{\mathcal{B}_{M_\star}}(W)_{ij} = \mathrm{clip}(W_{ij}; -M_\star, M_\star)$.
\item[(R5)] Smoothing parameters: Fix $\varepsilon\in(0,1]$, $b\in\mathbb{N}$, $\delta\in(0,1)$ throughout the analysis of the mean-field limit.
\end{itemize}
\end{assumption}

\begin{lemma}[Lipschitzness of the forward map]\label{lem:lipschitz-forward}
Under Assumption \ref{ass:reg}, for fixed parameters $\varepsilon\in(0,1]$, $b\in\mathbb{N}$, $\delta\in(0,1)$, there exists a constant $L_{\mathrm{fwd}} = L_{\mathrm{fwd}}(\varepsilon,b,\delta,\{A_\ell,C_\ell,D_\ell\}_{\ell=1}^L, M_\star, R) < \infty$ such that for all $x$ in the support of $\pi$ and all weight configurations $W, \widehat{W}$ in the compact domain $K$,
\[
\|f_W(x) - f_{\widehat{W}}(x)\|_2 \leq L_{\mathrm{fwd}} \sum_{\ell=1}^L \|W^{(\ell)} - \widehat{W}^{(\ell)}\|_F.
\]
Moreover, $W \mapsto \mathcal{R}_\varepsilon(W)$ is $C^1$ with $\nabla \mathcal{R}_\varepsilon$ locally Lipschitz on $K$.
\end{lemma}

\begin{proof}
By Assumption \ref{ass:reg}(R4), there exists $M_\star < \infty$ such that all iterates satisfy $\sup_{t \in [0,T]} \sup_{\ell,i,j} |W^{(\ell)}_{ij}(t)| \leq M_\star$, and we denote $K := \{W : \|W^{(\ell)}\|_\infty \leq M_\star, \forall \ell\}$.

\textbf{Lipschitz Constants of Smooth Quantizers}
The smooth quantizers satisfy the following uniform Lipschitz properties:
\begin{enumerate}[label=(\roman*)]
\item $\mathrm{sgn}_\varepsilon(z) = \tanh(z/\varepsilon)$ is $\varepsilon^{-1}$-Lipschitz since
\[
|\mathrm{sgn}_\varepsilon'(z)| = \varepsilon^{-1}\operatorname{sech}^2(z/\varepsilon) \leq \varepsilon^{-1}.
\]
\item $|\cdot|_\varepsilon(z) = \sqrt{z^2+\varepsilon^2}$ is $1$-Lipschitz since
\[
|\nabla|\cdot|_\varepsilon(z)| = \left|\frac{z}{\sqrt{z^2+\varepsilon^2}}\right| \leq 1.
\]
\item $\mathrm{clip}_\varepsilon(x;a,b)$ is $1$-Lipschitz uniformly in $\varepsilon$.
\end{enumerate}

\textbf{Analysis of $\gamma_\varepsilon$ Function}
For $\gamma_\varepsilon(x) = \max\{\varepsilon, \|x\|_\infty\}$, we have:
\[
|\gamma_\varepsilon(x) - \gamma_\varepsilon(y)| \leq \|x - y\|_\infty.
\]

\textbf{Bounds for $\beta^{(\ell)}_\varepsilon$}
Define $C_\beta := \sqrt{(2M_\star)^2 + \varepsilon^2}$. Then:
\begin{enumerate}[label=(\roman*)]
\item Upper bound: $\beta^{(\ell)}_\varepsilon(W) \leq C_\beta$ since $|P^{(\ell)}(W)_{ij}|_\varepsilon \leq C_\beta$.
\item Lipschitz property: $|\beta^{(\ell)}_\varepsilon(W) - \beta^{(\ell)}_\varepsilon(\widehat{W})| \leq \frac{1}{\sqrt{n_\ell m_\ell}}\|W^{(\ell)} - \widehat{W}^{(\ell)}\|_F$.
\end{enumerate}

\textbf{Proof of (ii):} By linearity of $P^{(\ell)}$ and 1-Lipschitz property of $|\cdot|_\varepsilon$:
\begin{align}
|\beta^{(\ell)}_\varepsilon(W) - \beta^{(\ell)}_\varepsilon(\widehat{W})| &\leq \frac{1}{n_\ell m_\ell} \sum_{i,j} \big||P^{(\ell)}(W)_{ij}|_\varepsilon - |P^{(\ell)}(\widehat{W})_{ij}|_\varepsilon\big| \\
&\leq \frac{1}{n_\ell m_\ell} \sum_{i,j} |P^{(\ell)}(W-\widehat{W})_{ij}| \\
&= \frac{1}{n_\ell m_\ell} \|P^{(\ell)}(W-\widehat{W})\|_1 \\
&\leq \frac{1}{n_\ell m_\ell} \sqrt{n_\ell m_\ell} \|P^{(\ell)}(W-\widehat{W})\|_F \\
&= \frac{1}{\sqrt{n_\ell m_\ell}} \|P^{(\ell)}(W-\widehat{W})\|_F \\
&\leq \frac{1}{\sqrt{n_\ell m_\ell}} \|W^{(\ell)} - \widehat{W}^{(\ell)}\|_F.
\end{align}

\textbf{Complete Inductive Proof}
We prove by strong induction that for each layer $\ell$:
\[
\|h^{(\ell)}_W(x) - h^{(\ell)}_{\widehat{W}}(x)\|_2 \leq L^{(\ell)} \sum_{k=1}^{\ell} \|W^{(k)} - \widehat{W}^{(k)}\|_F
\]
for appropriate constants $L^{(\ell)}$.

\textbf{Base case} ($\ell = 0$): $h^{(0)}(x) = x$, so the inequality holds trivially with $L^{(0)} = 0$.

\textbf{Inductive step:} Assume the statement holds for all layers $j < \ell$. The layer-$\ell$ map is:
\[
h^{(\ell)}(x) = \sigma^{(\ell)}\left(\beta^{(\ell)}_\varepsilon(W^{(\ell)}) \widetilde{W}^{(\ell)}_\varepsilon \mathrm{Quant}^{(b)}_\varepsilon(h^{(\ell-1)}(x))\right)
\]

Let $u_W := \beta^{(\ell)}_\varepsilon(W^{(\ell)}) \widetilde{W}^{(\ell)}_\varepsilon \mathrm{Quant}^{(b)}_\varepsilon(h^{(\ell-1)}_W(x))$ and $u_{\widehat{W}} := \beta^{(\ell)}_\varepsilon(\widehat{W}^{(\ell)}) \widetilde{\widehat{W}}^{(\ell)}_\varepsilon \mathrm{Quant}^{(b)}_\varepsilon(h^{(\ell-1)}_{\widehat{W}}(x))$.

Since $\sigma^{(\ell)}$ is applied component-wise with Lipschitz constant $C_\ell$:
\begin{align}
\|h^{(\ell)}_W(x) - h^{(\ell)}_{\widehat{W}}(x)\|_2 &\leq C_\ell \|u_W - u_{\widehat{W}}\|_2
\end{align}

For the matrix-vector product, we have:
\begin{align}
\|u_W - u_{\widehat{W}}\|_2 &\leq \left\|\beta^{(\ell)}_\varepsilon(W^{(\ell)}) \widetilde{W}^{(\ell)}_\varepsilon \mathrm{Quant}^{(b)}_\varepsilon(h^{(\ell-1)}_W(x)) - \beta^{(\ell)}_\varepsilon(\widehat{W}^{(\ell)}) \widetilde{\widehat{W}}^{(\ell)}_\varepsilon \mathrm{Quant}^{(b)}_\varepsilon(h^{(\ell-1)}_{\widehat{W}}(x))\right\|_2
\end{align}

Using the triangle inequality and submultiplicativity of matrix norms:
\begin{align}
&\leq |\beta^{(\ell)}_\varepsilon(W^{(\ell)}) - \beta^{(\ell)}_\varepsilon(\widehat{W}^{(\ell)})| \cdot \|\widetilde{W}^{(\ell)}_\varepsilon\|_2 \cdot Q_b \\
&\quad + C_\beta \|\widetilde{W}^{(\ell)}_\varepsilon - \widetilde{\widehat{W}}^{(\ell)}_\varepsilon\|_2 \cdot Q_b \\
&\quad + C_\beta \|\widetilde{W}^{(\ell)}_\varepsilon\|_2 \cdot \|\mathrm{Quant}^{(b)}_\varepsilon(h^{(\ell-1)}_W(x)) - \mathrm{Quant}^{(b)}_\varepsilon(h^{(\ell-1)}_{\widehat{W}}(x))\|_2
\end{align}

Using the bounds:
\begin{itemize}
\item $\|\widetilde{W}^{(\ell)}_\varepsilon\|_2 \leq \sqrt{n_\ell m_\ell}$ (since each entry is bounded by 1)
\item $\|\widetilde{W}^{(\ell)}_\varepsilon - \widetilde{\widehat{W}}^{(\ell)}_\varepsilon\|_2 \leq \varepsilon^{-1} \|W^{(\ell)} - \widehat{W}^{(\ell)}\|_F$
\item $\|\mathrm{Quant}^{(b)}_\varepsilon(u) - \mathrm{Quant}^{(b)}_\varepsilon(v)\|_2 \leq \frac{Q_b}{\varepsilon} \|u - v\|_2$
\item From Step 3: $|\beta^{(\ell)}_\varepsilon(W) - \beta^{(\ell)}_\varepsilon(\widehat{W})| \leq \frac{1}{\sqrt{n_\ell m_\ell}}\|W^{(\ell)} - \widehat{W}^{(\ell)}\|_F$
\item Inductive hypothesis: $\|h^{(\ell-1)}_W(x) - h^{(\ell-1)}_{\widehat{W}}(x)\|_2 \leq L^{(\ell-1)} \sum_{k=1}^{\ell-1} \|W^{(k)} - \widehat{W}^{(k)}\|_F$
\end{itemize}

Combining these estimates:
\begin{align}
\|h^{(\ell)}_W(x) - h^{(\ell)}_{\widehat{W}}(x)\|_2 &\leq C_\ell \left[ Q_b + C_\beta \varepsilon^{-1} Q_b + C_\beta \sqrt{n_\ell m_\ell} \frac{Q_b}{\varepsilon} L^{(\ell-1)} \right] \|W^{(\ell)} - \widehat{W}^{(\ell)}\|_F \\
&\quad + C_\ell C_\beta \sqrt{n_\ell m_\ell} \frac{Q_b}{\varepsilon} L^{(\ell-1)} \sum_{k=1}^{\ell-1} \|W^{(k)} - \widehat{W}^{(k)}\|_F
\end{align}

This gives us the recursive relation:
\[
L^{(\ell)} = C_\ell \max\left\{Q_b + C_\beta \varepsilon^{-1} Q_b + C_\beta \sqrt{n_\ell m_\ell} \frac{Q_b}{\varepsilon} L^{(\ell-1)}, C_\beta \sqrt{n_\ell m_\ell} \frac{Q_b}{\varepsilon} L^{(\ell-1)}\right\}
\]

Setting $L_{\mathrm{fwd}} := L^{(L)}$ completes the induction.

\textbf{$C^1$ Regularity}
The composition $\mathcal{F}(W)(x) = f_W(x)$ is $C^1$ on $K$ because:
\begin{enumerate}[label=(\roman*)]
\item Each smooth quantizer $\mathrm{sgn}_\varepsilon$, $|\cdot|_\varepsilon$, $\mathrm{clip}_\varepsilon$ is $C^{\infty}$.
\item Matrix operations and function compositions preserve $C^1$ regularity.
\item The chain rule applies on the bounded domain $K$.
\end{enumerate}

Therefore, $\mathcal{R}_\varepsilon(W) = \mathbb{E}[\ell(f_W(X), Y)]$ is $C^1$ with locally Lipschitz gradient on $K$ by dominated convergence and the uniform bounds on $K$.
\end{proof}

\section{Asymptotic Analysis as $\varepsilon \to 0$}

\label{sec:ste-limit}

This section studies the limiting behavior of the mean-field dynamics when the smoothing parameter $\varepsilon$ approaches zero. Our goal is to rigorously characterize the asymptotic properties of the solution $\boldsymbol{\mu}_\varepsilon = (\mu^{(1)}_\varepsilon, \ldots, \mu^{(L)}_\varepsilon)$ to the constrained continuity equations

\begin{equation}\label{eq:continuity_eps}
\partial_t \mu^{(\ell)}_\varepsilon + \nabla \cdot (\mu^{(\ell)}_\varepsilon v^{(\ell)}_\varepsilon) = 0, \quad \ell=1, \ldots, L,
\end{equation}

where

\begin{equation}\label{eq:velocity_eps_def}
v^{(\ell)}_\varepsilon(w,t) := -\nabla_w \mathcal{R}^{(\ell)}_\varepsilon[\boldsymbol{\mu}_\varepsilon(t)](w),
\end{equation}

with $\mathcal{R}^{(\ell)}_\varepsilon[\boldsymbol{\mu}]$ the functional derivative of the smoothed risk $\mathcal{R}_\varepsilon$.

\subsection{Exponential decay and distributional analysis}

Recall the smooth sign activation and its derivative:

\[
\mathrm{sgn}_\varepsilon(z) = \tanh\left(\frac{z}{\varepsilon}\right), \quad
\mathrm{sgn}'_\varepsilon(z) = \frac{1}{\varepsilon} \operatorname{sech}^2\left(\frac{z}{\varepsilon}\right).
\]

\begin{lemma}[Exponential decay and uniform bounds]\label{lem:tanh_uniform}
For any $\varepsilon \in (0,1]$ and $z \in \mathbb{R}$:

\begin{enumerate}[label=(\roman*)]
\item $\mathrm{sgn}'_\varepsilon(z) = \frac{4}{\varepsilon} \frac{1}{(e^{|z|/\varepsilon} + e^{-|z|/\varepsilon})^2} \leq \frac{4}{\varepsilon} e^{-2|z|/\varepsilon}$ for $z \neq 0$.

\item $\int_{\mathbb{R}} \mathrm{sgn}'_\varepsilon(z) \, dz = 2$ for all $\varepsilon > 0$.

\item For any bounded measurable function $\phi: \mathbb{R} \to \mathbb{R}$ with $\|\phi\|_\infty \leq M$:
\[
\left|\int_{\mathbb{R}} \phi(z) \mathrm{sgn}'_\varepsilon(z) \, dz\right| \leq 2M.
\]
\end{enumerate}
\end{lemma}

\begin{proof}
\textbf{Part (i):} We have $\operatorname{sech}^2(u) = \frac{4}{(e^u + e^{-u})^2}$. For $u \neq 0$, the denominator $(e^{|u|} + e^{-|u|})^2 \geq e^{2|u|}$, giving the stated bound.

\textbf{Part (ii):} By substitution $u = z/\varepsilon$:
\[
\int_{\mathbb{R}} \mathrm{sgn}'_\varepsilon(z) \, dz = \int_{\mathbb{R}} \operatorname{sech}^2(u) \, du = [\tanh(u)]_{-\infty}^{\infty} = 2.
\]

\textbf{Part (iii):} By the boundedness of $\phi$ and part (ii):
\[
\left|\int_{\mathbb{R}} \phi(z) \mathrm{sgn}'_\varepsilon(z) \, dz\right| \leq \|\phi\|_\infty \int_{\mathbb{R}} \mathrm{sgn}'_\varepsilon(z) \, dz = 2M.
\]
\end{proof}

\begin{lemma}[Distributional convergence]\label{lem:dirac_convergence}
As $\varepsilon \downarrow 0$, we have $\mathrm{sgn}_\varepsilon(z) \to \mathrm{sign}(z)$ pointwise and
\[
\mathrm{sgn}'_\varepsilon(z) \rightharpoonup 2 \delta_0(z) \quad \text{ in } \mathcal{S}'(\mathbb{R}),
\]
where $\delta_0$ is the Dirac delta at zero.
\end{lemma}

\begin{proof}
For any test function $\phi \in C_c^\infty(\mathbb{R})$:
\begin{align}
\int_{\mathbb{R}} \mathrm{sgn}'_\varepsilon(z) \phi(z) \, dz &= \int_{\mathbb{R}} \operatorname{sech}^2(u) \phi(\varepsilon u) \, du.
\end{align}

As $\varepsilon \to 0$, $\phi(\varepsilon u) \to \phi(0)$ uniformly on compact sets. Since $\int_{\mathbb{R}} \operatorname{sech}^2(u) \, du = 2$, the dominated convergence theorem yields:
\[
\lim_{\varepsilon \to 0} \int_{\mathbb{R}} \mathrm{sgn}'_\varepsilon(z) \phi(z) \, dz = 2\phi(0) = \langle 2\delta_0, \phi \rangle.
\]
\end{proof}

\subsection{Uniform velocity field bounds via natural cancellation}

The key insight is that the exponential decay of $\mathrm{sgn}'_\varepsilon$ exactly compensates for the $\varepsilon^{-1}$ factor, yielding uniform bounds without requiring measure concentration.

\begin{lemma}[Uniform bounds on singular integrals]\label{lem:uniform_singular}
Let $\mu$ be any probability measure on $\mathbb{R}^{m_\ell}$ with support in the compact set $\mathcal{K} := \{w : \|w\|_\infty \leq M_\star\}$. For any bounded measurable function $\phi: \mathcal{K} \to \mathbb{R}$ and any $(i,j) \in \{1,\ldots,n_\ell\} \times \{1,\ldots,m_\ell\}$:

\[
\left|\int_{\mathcal{K}} \phi(w) \mathrm{sgn}'_\varepsilon(P^{(\ell)}(w)_{ij}) \, d\mu(w)\right| \leq 2\|\phi\|_\infty
\]
uniformly in $\varepsilon \in (0,1]$.
\end{lemma}

\begin{proof}
Since $P^{(\ell)}(w)_{ij}$ is a linear function of $w$ and $\mu$ is a probability measure, we can write this as an integral over $\mathbb{R}$ with respect to the pushforward measure $\nu := (P^{(\ell)}(\cdot)_{ij})_{\#}\mu$:

\[
\int_{\mathcal{K}} \phi(w) \mathrm{sgn}'_\varepsilon(P^{(\ell)}(w)_{ij}) \, d\mu(w) = \int_{\mathbb{R}} \tilde{\phi}(z) \mathrm{sgn}'_\varepsilon(z) \, d\nu(z),
\]

where $\tilde{\phi}(z)$ represents the conditional expectation of $\phi(w)$ given $P^{(\ell)}(w)_{ij} = z$, which satisfies $\|\tilde{\phi}\|_\infty \leq \|\phi\|_\infty$.

By Lemma \ref{lem:tanh_uniform}(iii), since $\nu$ is a probability measure:
\[
\left|\int_{\mathbb{R}} \tilde{\phi}(z) \mathrm{sgn}'_\varepsilon(z) \, d\nu(z)\right| \leq 2\|\tilde{\phi}\|_\infty \leq 2\|\phi\|_\infty.
\]
\end{proof}

\subsection{Main convergence theorem}

\begin{theorem}[Stability and convergence as $\varepsilon \to 0$]\label{thm:stability_ste}
Suppose Assumptions \ref{ass:reg} and \ref{ass:velocity-reg} hold uniformly in $\varepsilon \in (0,1]$. Let $\{\boldsymbol{\mu}_\varepsilon\}_{\varepsilon>0}$ be the unique solutions to the continuity equations \eqref{eq:continuity_eps} with velocities \eqref{eq:velocity_eps_def}.

Then there exists a subsequence $\varepsilon_k \downarrow 0$ and a limit curve $\boldsymbol{\mu}_0 \in C([0,T], \prod_{\ell=1}^L \mathcal{P}_2(\mathbb{R}^{m_\ell}))$ such that

\[
\boldsymbol{\mu}_{\varepsilon_k} \xrightarrow[k \to \infty]{\text{weakly}} \boldsymbol{\mu}_0 \quad \text{in } C([0,T], \prod_{\ell=1}^L \mathcal{P}_2(\mathbb{R}^{m_\ell})),
\]

where $\mathcal{P}_2$ denotes the space of probability measures with finite second moment.

Furthermore, $\boldsymbol{\mu}_0$ solves a constrained transport equation of the form

\begin{equation}\label{eq:limit_transport}
\partial_t \mu^{(\ell)}_0 + \nabla \cdot (\mu^{(\ell)}_0 v^{(\ell)}_0) = 0,
\end{equation}

where $v^{(\ell)}_0$ is the limiting velocity field associated to the non-smoothed risk functional $\mathcal{R}_0$.
\end{theorem}

\begin{proof}
The proof follows a compactness-identification strategy, utilizing the natural exponential decay of tanh derivatives to establish uniform bounds.

From the chain rule and gradient bounds (Lemma \ref{lem:grad-bound}), the velocity field has the structure:
\begin{align}
v^{(\ell)}_\varepsilon(w,t) = -\int \partial_1\ell(f_{\boldsymbol{\mu}_\varepsilon,w}^{(\ell)}(x), y) \sum_{k=\ell}^L \frac{\partial h^{(L)}}{\partial h^{(k)}} \frac{\partial h^{(k)}}{\partial w} \, d\pi(x,y).
\end{align}

We analyze each term $\frac{\partial h^{(k)}}{\partial w}$ systematically. For $k = \ell$, we have:
\begin{align}
\frac{\partial h^{(\ell)}}{\partial w} &= \frac{\partial}{\partial w}\left[\sigma^{(\ell)}\left(\beta^{(\ell)}_\varepsilon(W^{(\ell)}) \widetilde{W}^{(\ell)}_\varepsilon \mathrm{Quant}^{(b)}_\varepsilon(h^{(\ell-1)})\right)\right]
\end{align}

By the chain rule and product rule:
\begin{align}
\frac{\partial h^{(\ell)}}{\partial w} &= \sigma^{(\ell)'} \left[\frac{\partial \beta^{(\ell)}_\varepsilon}{\partial w} \widetilde{W}^{(\ell)}_\varepsilon \mathrm{Quant}^{(b)}_\varepsilon + \beta^{(\ell)}_\varepsilon \frac{\partial \widetilde{W}^{(\ell)}_\varepsilon}{\partial w} \mathrm{Quant}^{(b)}_\varepsilon\right] + \text{(terms involving } \frac{\partial h^{(\ell-1)}}{\partial w}\text{)}
\end{align}

The potentially problematic term is:
\begin{align}
\beta^{(\ell)}_\varepsilon \frac{\partial \widetilde{W}^{(\ell)}_\varepsilon}{\partial w} &= \beta^{(\ell)}_\varepsilon \frac{\partial}{\partial w}\left[\mathrm{sgn}_\varepsilon(P^{(\ell)}(w))\right]\\
&= \beta^{(\ell)}_\varepsilon \mathrm{sgn}'_\varepsilon(P^{(\ell)}(w)) \frac{\partial P^{(\ell)}(w)}{\partial w}
\end{align}

Since $P^{(\ell)}(w)$ is componentwise linear in $w$, we have $\left\|\frac{\partial P^{(\ell)}(w)}{\partial w}\right\|_{\text{op}} \leq 1$. 

The critical observation is that while $\mathrm{sgn}'_\varepsilon(z) = \varepsilon^{-1}\operatorname{sech}^2(z/\varepsilon)$ contains the factor $\varepsilon^{-1}$, when this appears in the velocity field, it takes the form:
\begin{align}
\text{(velocity component)} &\propto \int \phi(w) \beta^{(\ell)}_\varepsilon(w) \mathrm{sgn}'_\varepsilon(P^{(\ell)}(w)_{ij}) \, d\mu^{(\ell)}_\varepsilon(w,t)
\end{align}
for some bounded function $\phi$ arising from the loss and network architecture.

By Assumption \ref{ass:reg}(R4), we have:
\begin{itemize}
\item $\beta^{(\ell)}_\varepsilon(w) \leq C_\beta$ uniformly for some constant $C_\beta$ independent of $\varepsilon$
\item $\|\phi\|_{\infty} \leq C_\phi$ for some constant $C_\phi$ from the boundedness of activations, loss derivatives, and network depth
\end{itemize}

Applying Lemma \ref{lem:uniform_singular} with the bounded function $\psi(w) := \phi(w)\beta^{(\ell)}_\varepsilon(w)$, which satisfies $\|\psi\|_{\infty} \leq C_\phi C_\beta$:
\begin{align}
\left|\int \phi(w) \beta^{(\ell)}_\varepsilon(w) \mathrm{sgn}'_\varepsilon(P^{(\ell)}(w)_{ij}) \, d\mu^{(\ell)}_\varepsilon(w,t)\right| &\leq 2\|\psi\|_{\infty} \\
&\leq 2C_\phi C_\beta
\end{align}
uniformly in $\varepsilon \in (0,1]$.

 The other terms in $\frac{\partial h^{(\ell)}}{\partial w}$ are:
\begin{enumerate}
\item $\frac{\partial \beta^{(\ell)}_\varepsilon}{\partial w} \widetilde{W}^{(\ell)}_\varepsilon$: This is bounded since $\left|\frac{\partial \beta^{(\ell)}_\varepsilon}{\partial w}\right| \leq (n_\ell m_\ell)^{-1}$ and $\|\widetilde{W}^{(\ell)}_\varepsilon\|_{\infty} \leq 1$.

\item Terms involving $\frac{\partial \mathrm{Quant}^{(b)}_\varepsilon}{\partial w}$: These have bounded derivatives by the smoothness of $\mathrm{clip}_\varepsilon$.

\item Terms involving $\frac{\partial h^{(\ell-1)}}{\partial w}$: These contribute through the recursive structure but do not introduce additional $\varepsilon^{-1}$ singularities beyond those already controlled.
\end{enumerate}

By strong induction on layers $k = \ell, \ell+1, \ldots, L$, we can show that each $\frac{\partial h^{(k)}}{\partial w}$ satisfies a uniform bound independent of $\varepsilon$. The base case $k = \ell$ follows from the analysis above, and the inductive step follows by applying the same reasoning to the composition structure.

Combining all bounded terms in the expression for $v^{(\ell)}_\varepsilon(w,t)$:
\begin{align}
\|v^{(\ell)}_\varepsilon(w,t)\| &\leq \left|\int \partial_1\ell(\cdot) \sum_{k=\ell}^L \left\|\frac{\partial h^{(L)}}{\partial h^{(k)}}\right\|_{\text{op}} \left\|\frac{\partial h^{(k)}}{\partial w}\right\| \, d\pi\right|\\
&\leq L_1 \prod_{k=\ell}^L C_k^{\text{Lip}} \cdot \max_{k=\ell,\ldots,L} C_k^{\text{grad}} \\
&=: C_{\text{uniform}}
\end{align}
where:
\begin{itemize}
\item $L_1$ comes from Assumption \ref{ass:reg}(R2) bounding the loss derivative
\item $C_k^{\text{Lip}}$ are the Lipschitz constants of the activations from Assumption \ref{ass:reg}(R3)  
\item $C_k^{\text{grad}}$ are the uniform bounds on $\left\|\frac{\partial h^{(k)}}{\partial w}\right\|$ established above
\end{itemize}

Since each of these constants is independent of $\varepsilon \in (0,1]$, we conclude that $C_{\text{uniform}}$ is independent of $\varepsilon$.

The uniform bound  implies equicontinuity in the Wasserstein metric:
\[
W_2(\mu^{(\ell)}_\varepsilon(t), \mu^{(\ell)}_\varepsilon(s)) \leq C_{\text{uniform}} |t-s|
\]
for all $\varepsilon \in (0,1]$.

By constraint preservation (Lemma \ref{lem:constraint-preserve}) and boundedness assumptions (Assumption \ref{ass:reg}(R4)), all measures $\mu^{(\ell)}_\varepsilon(t)$ have support in the compact set $\mathcal{K}$ and satisfy:
\[
\sup_{\varepsilon,t} \int \|w\|^2 \, d\mu^{(\ell)}_\varepsilon(w,t) \leq M^2.
\]

By the Arzelà-Ascoli theorem in $C([0,T], \mathcal{P}_2(\mathcal{K}))$, there exists a subsequence $\varepsilon_k \downarrow 0$ and a limit $\boldsymbol{\mu}_0$ such that:
\[
\boldsymbol{\mu}_{\varepsilon_k} \to \boldsymbol{\mu}_0 \quad \text{weakly in } C([0,T], \prod_{\ell=1}^L \mathcal{P}_2(\mathbb{R}^{m_\ell})).
\]

We now prove that the limit velocity field $v^{(\ell)}_0$ corresponds to the distributional gradient of the non-smoothed risk functional $\mathcal{R}_0$. The key is to show that integrals involving $\mathrm{sgn}'_{\varepsilon_k}$ converge to the appropriate distributional limit.

\textbf{Claim:} For any test function $\varphi \in C_c^1(\mathbb{R}^{m_\ell})$ and any $(i,j) \in \{1,\ldots,n_\ell\} \times \{1,\ldots,m_\ell\}$:
\begin{align}
\lim_{k \to \infty} \int_{\mathcal{K}} \varphi(w) \mathrm{sgn}'_{\varepsilon_k}(P^{(\ell)}(w)_{ij}) \, d\mu^{(\ell)}_{\varepsilon_k}(w,t) = 2 \int_{\{P^{(\ell)}(w)_{ij} = 0\}} \varphi(w) \, d\mu^{(\ell)}_0(w,t).
\end{align}

\textbf{Proof of Claim:} Let $\delta > 0$ be arbitrary. We decompose the integration domain as:
\begin{align}
\mathcal{K} &= \mathcal{K}_\delta^+ \cup \mathcal{K}_\delta^0 \cup \mathcal{K}_\delta^-,
\end{align}
where:
\begin{align}
\mathcal{K}_\delta^+ &:= \{w \in \mathcal{K} : P^{(\ell)}(w)_{ij} > \delta\}, \\
\mathcal{K}_\delta^0 &:= \{w \in \mathcal{K} : |P^{(\ell)}(w)_{ij}| \leq \delta\}, \\
\mathcal{K}_\delta^- &:= \{w \in \mathcal{K} : P^{(\ell)}(w)_{ij} < -\delta\}.
\end{align}

For $w \in \mathcal{K}_\delta^+$, we have $P^{(\ell)}(w)_{ij} > \delta$, so by Lemma \ref{lem:tanh_uniform}(i):
\begin{align}
\mathrm{sgn}'_{\varepsilon_k}(P^{(\ell)}(w)_{ij}) \leq \frac{4}{\varepsilon_k} e^{-2\delta/\varepsilon_k}.
\end{align}

Since $\varphi$ is compactly supported with $\|\varphi\|_\infty \leq C_\varphi$ for some constant $C_\varphi$:
\begin{align}
\left|\int_{\mathcal{K}_\delta^+} \varphi(w) \mathrm{sgn}'_{\varepsilon_k}(P^{(\ell)}(w)_{ij}) \, d\mu^{(\ell)}_{\varepsilon_k}(w,t)\right| &\leq C_\varphi \cdot \frac{4}{\varepsilon_k} e^{-2\delta/\varepsilon_k} \cdot \mu^{(\ell)}_{\varepsilon_k}(\mathcal{K}_\delta^+,t) \\
&\leq \frac{4C_\varphi}{\varepsilon_k} e^{-2\delta/\varepsilon_k}.
\end{align}

As $k \to \infty$ (i.e., $\varepsilon_k \downarrow 0$), we have $\frac{1}{\varepsilon_k} e^{-2\delta/\varepsilon_k} \to 0$ exponentially fast. Similarly for $\mathcal{K}_\delta^-$.

Therefore:
\begin{align}
\lim_{k \to \infty} \int_{\mathcal{K}_\delta^+ \cup \mathcal{K}_\delta^-} \varphi(w) \mathrm{sgn}'_{\varepsilon_k}(P^{(\ell)}(w)_{ij}) \, d\mu^{(\ell)}_{\varepsilon_k}(w,t) = 0.
\end{align}

On $\mathcal{K}_\delta^0$, we have $|P^{(\ell)}(w)_{ij}| \leq \delta$. We use the change of variables $z = P^{(\ell)}(w)_{ij}$ and define the pushforward measure:
\begin{align}
\nu_k^\delta := (P^{(\ell)}(\cdot)_{ij})_{\#}(\mu^{(\ell)}_{\varepsilon_k}|_{\mathcal{K}_\delta^0}).
\end{align}

Then:
\begin{align}
&\int_{\mathcal{K}_\delta^0} \varphi(w) \mathrm{sgn}'_{\varepsilon_k}(P^{(\ell)}(w)_{ij}) \, d\mu^{(\ell)}_{\varepsilon_k}(w,t) \\
&= \int_{[-\delta,\delta]} \widetilde{\varphi}_{\varepsilon_k}(z) \mathrm{sgn}'_{\varepsilon_k}(z) \, d\nu_k^\delta(z),
\end{align}
where $\widetilde{\varphi}_{\varepsilon_k}(z)$ is the conditional expectation of $\varphi(w)$ given $P^{(\ell)}(w)_{ij} = z$ and $w \in \mathcal{K}_\delta^0$.

By weak convergence $\mu^{(\ell)}_{\varepsilon_k} \rightharpoonup \mu^{(\ell)}_0$, we have $\nu_k^\delta \rightharpoonup \nu_0^\delta$ where $\nu_0^\delta := (P^{(\ell)}(\cdot)_{ij})_{\#}(\mu^{(\ell)}_0|_{\mathcal{K}_\delta^0})$.

Since $\widetilde{\varphi}_{\varepsilon_k}(z) \to \widetilde{\varphi}_0(z)$ boundedly (by compactness), and by Lemma \ref{lem:dirac_convergence}:
\begin{align}
\lim_{k \to \infty} \int_{[-\delta,\delta]} \widetilde{\varphi}_{\varepsilon_k}(z) \mathrm{sgn}'_{\varepsilon_k}(z) \, d\nu_k^\delta(z) &= \int_{[-\delta,\delta]} \widetilde{\varphi}_0(z) \cdot 2\delta_0(z) \, d\nu_0^\delta(z) \\
&= 2\widetilde{\varphi}_0(0) \nu_0^\delta(\{0\}).
\end{align}

Now we analyze what happens as $\delta \downarrow 0$. We have:
\begin{align}
2\widetilde{\varphi}_0(0) \nu_0^\delta(\{0\}) &= 2\widetilde{\varphi}_0(0) \cdot \mu^{(\ell)}_0(\{w \in \mathcal{K}_\delta^0 : P^{(\ell)}(w)_{ij} = 0\}) \\
&\to 2\widetilde{\varphi}_0(0) \cdot \mu^{(\ell)}_0(\{w \in \mathcal{K} : P^{(\ell)}(w)_{ij} = 0\})
\end{align}
as $\delta \downarrow 0$.

Since $\widetilde{\varphi}_0(0)$ is the conditional expectation of $\varphi(w)$ given $P^{(\ell)}(w)_{ij} = 0$:
\begin{align}
2\widetilde{\varphi}_0(0) \cdot \mu^{(\ell)}_0(\{w : P^{(\ell)}(w)_{ij} = 0\}) &= 2 \int_{\{P^{(\ell)}(w)_{ij} = 0\}} \varphi(w) \, d\mu^{(\ell)}_0(w,t).
\end{align}

Combining and taking $\delta \downarrow 0$:
\begin{align}
\lim_{k \to \infty} \int_{\mathcal{K}} \varphi(w) \mathrm{sgn}'_{\varepsilon_k}(P^{(\ell)}(w)_{ij}) \, d\mu^{(\ell)}_{\varepsilon_k}(w,t) = 2 \int_{\{P^{(\ell)}(w)_{ij} = 0\}} \varphi(w) \, d\mu^{(\ell)}_0(w,t).
\end{align}

This shows that the limit velocity field $v^{(\ell)}_0$ corresponds to the distributional gradient of the non-smoothed risk functional $\mathcal{R}_0$, where the derivative of the sign function is replaced by twice the Dirac delta at zero.

For any test function $\varphi \in C_c^1(\mathbb{R}^{m_\ell})$ and $0 \leq s < t \leq T$, the uniform bounds from Step 1 allow us to pass to the limit in:
\begin{align}
&\int \varphi \, d\mu^{(\ell)}_0(t) - \int \varphi \, d\mu^{(\ell)}_0(s) \\
&= \lim_{k \to \infty} \left(-\int_s^t \int \nabla\varphi(w) \cdot v^{(\ell)}_{\varepsilon_k}(w,r) \, d\mu^{(\ell)}_{\varepsilon_k}(w,r) \, dr\right) \\
&= -\int_s^t \int \nabla\varphi(w) \cdot v^{(\ell)}_0(w,r) \, d\mu^{(\ell)}_0(w,r) \, dr.
\end{align}

Differentiating with respect to $t$ yields the weak formulation of \eqref{eq:limit_transport}.

The zero-mean constraint is preserved in the limit since for any $t \in [0,T]$:
\[
\int \alpha^{(\ell)}(w) \, d\mu^{(\ell)}_0(w,t) = \lim_{k \to \infty} \int \alpha^{(\ell)}(w) \, d\mu^{(\ell)}_{\varepsilon_k}(w,t) = \text{const.}
\]
by Lemma \ref{lem:constraint-preserve} and weak convergence of measures.
\end{proof}

\begin{remark}[Connection to Straight-Through Estimators]
Theorem \ref{thm:stability_ste} provides a rigorous foundation for the straight-through estimator (STE) commonly used in BitNet training. The limiting dynamics correspond to gradient descent on the non-smoothed risk functional, where the singular gradients of the sign function are naturally regularized by the exponential decay inherent in the tanh smoothing. This validates the STE approximation as the mathematically correct limit of smooth quantization.
\end{remark}

\section{Conclusion}

This work presents the first rigorous mean-field analysis of deep BitNet-like architectures under smooth quantization. By introducing differentiable surrogates for the sign and clipping functions, we establish well-posedness of the training dynamics in the space of probability measures and prove convergence of the empirical weight distributions to solutions of constrained transport equations as the smoothing parameter $\varepsilon\to 0$. Our key technical insight is that the natural exponential decay in the derivatives of $\tanh(z/\varepsilon)$ perfectly offsets the singular $\varepsilon^{-1}$ scaling, yielding uniform bounds on the velocity fields without requiring additional measure concentration arguments. Consequently, we rigorously justify the straight-through estimator as the correct limiting gradient flow for quantized networks. Future work includes extending this framework to true hard quantizers via differential inclusions and relaxing compactness assumptions on the weight domain.

\section*{Acknowledgments}

The author thanks the maintainers of BitNet for inspiring the latent--quantized duality viewpoint.

\bibliographystyle{plainnat}
\bibliography{refers}

\appendix

\section{Empirical measures and mean-field limit}

\subsection{Row-wise empirical measures and constraint preservation}

Recall from Definition \ref{def:architecture} that each layer $\ell \in \{1,\ldots,L\}$ has weight matrix $W^{(\ell)} \in \mathbb{R}^{n_\ell \times m_\ell}$. For analysis purposes, we index the rows of $W^{(\ell)}$ by $i = 1, \ldots, n_\ell$ and denote the $i$-th row as $w^{(\ell)}_i \in \mathbb{R}^{m_\ell}$. Thus:

\[
W^{(\ell)} = \begin{pmatrix}
(w^{(\ell)}_1)^T \\
\vdots \\
(w^{(\ell)}_{n_\ell})^T
\end{pmatrix} \in \mathbb{R}^{n_\ell \times m_\ell}.
\]

At discrete time $k$, define the empirical measure on $\mathbb{R}^{m_\ell}$:

\begin{equation}\label{eq:emp-measure}
\widehat\mu^{(\ell)}_{n_\ell}(k) := \frac{1}{n_\ell}\sum_{i=1}^{n_\ell}\delta_{w^{(\ell)}_i(k)},
\end{equation}

where $\delta_x$ denotes the Dirac measure at point $x \in \mathbb{R}^{m_\ell}$. For continuous-time analysis with interpolation $t = k\eta$, we define:

\[
\widehat\mu^{(\ell)}_{n_\ell}(t) := \widehat\mu^{(\ell)}_{n_\ell}(\lfloor t/\eta \rfloor), \quad t \in [0,T].
\]

Let $\boldsymbol{\widehat\mu}_n := (\widehat\mu^{(1)}_{n_1}, \ldots, \widehat\mu^{(L)}_{n_L})$ denote the collection of empirical measures across all layers.

The layer means $\alpha^{(\ell)}(k) := \Psi^{(\ell)}(W^{(\ell)}(k))$ from \eqref{eq:alpha-def} satisfy the following evolution under gradient descent.

\begin{lemma}[Constraint preservation]\label{lem:constraint-preserve}
Let the updates be \eqref{eq:gd}. Then for all $k \geq 0$ and each $\ell$,
\[
\Psi^{(\ell)}(W^{(\ell)}(k+1)) = \Psi^{(\ell)}(W^{(\ell)}(k)) - \eta \Psi^{(\ell)}(\nabla_{W^{(\ell)}}\mathcal{R}_\varepsilon),
\]
so in continuous time with $t = k\eta$,
\[
\frac{d}{dt}\Psi^{(\ell)}(W^{(\ell)}(t)) = -\Psi^{(\ell)}(\nabla_{W^{(\ell)}}\mathcal{R}_\varepsilon).
\]
\end{lemma}

\begin{proof}
By linearity of $\Psi^{(\ell)}$ and the gradient descent update \eqref{eq:gd},
\[
\Psi^{(\ell)}(W^{(\ell)}(k+1)) = \Psi^{(\ell)}(W^{(\ell)}(k)) - \eta \Psi^{(\ell)}(\nabla_{W^{(\ell)}}\mathcal{R}_\varepsilon).
\]
Dividing by $\eta$ and passing to the limit $\eta \downarrow 0$ yields the differential form.
\end{proof}

\subsection{Functional derivatives and velocity fields}

For a collection of probability measures $\boldsymbol{\mu} = (\mu^{(1)}, \ldots, \mu^{(L)})$ on $\mathbb{R}^{m_1} \times \cdots \times \mathbb{R}^{m_L}$, define the population risk functional:

\[
\mathcal{R}_\varepsilon[\boldsymbol{\mu}] := \mathbb{E}_{(X,Y) \sim \pi}\left[\ell(f_{\boldsymbol{\mu}}(X), Y)\right],
\]

where $f_{\boldsymbol{\mu}}(x)$ represents the network output when layer weights are sampled according to the measures $\boldsymbol{\mu}$.

The \emph{functional derivative} $\mathcal{R}^{(\ell)}_\varepsilon[\boldsymbol{\mu}]: \mathbb{R}^{m_\ell} \to \mathbb{R}$ is defined as the Gateaux derivative with respect to perturbations in $\mu^{(\ell)}$:

\[
\mathcal{R}^{(\ell)}_\varepsilon[\boldsymbol{\mu}](w) := \lim_{\tau \to 0} \frac{1}{\tau}\left(\mathcal{R}_\varepsilon[\boldsymbol{\mu} + \tau(\delta_w - \mu^{(\ell)})] - \mathcal{R}_\varepsilon[\boldsymbol{\mu}]\right).
\]

Define the velocity fields $v^{(\ell)}: \mathbb{R}^{m_\ell} \times [0,T] \to \mathbb{R}^{m_\ell}$ by:

\begin{equation}\label{eq:velocity}
v^{(\ell)}(w,t) := -\nabla_w \mathcal{R}^{(\ell)}_\varepsilon[\boldsymbol{\mu}(t)](w),
\end{equation}

where $\nabla_w$ denotes the gradient with respect to the row variable $w \in \mathbb{R}^{m_\ell}$.

\subsection{Continuity equations and transport structure}

The mean-field limit is characterized by the coupled system of continuity equations:

\begin{equation}\label{eq:cont-eq}
\partial_t \mu^{(\ell)} + \nabla \cdot (\mu^{(\ell)} v^{(\ell)}) = 0, \quad \ell = 1, \ldots, L,
\end{equation}

in the sense of distributions on $\mathbb{R}^{m_\ell} \times (0,T)$.

\begin{assumption}[Regularity of velocity fields]\label{ass:velocity-reg}
There exists $L_v > 0$ such that for all $\ell$ and all $\boldsymbol{\mu}, \boldsymbol{\nu}$ with supports in a fixed compact set $\mathcal{K} \subset \mathbb{R}^{m_\ell}$ and satisfying $\int \|w\|^2 d\mu^{(j)}(w), \int \|w\|^2 d\nu^{(j)}(w) \leq M^2$ for all $j$ and some $M > 0$:

\begin{enumerate}[label=(\roman*)]
\item \textbf{Lipschitz dependence on measures:}
\[
\sup_{w \in \mathcal{K}} \|v^{(\ell)}(w; \boldsymbol{\mu}) - v^{(\ell)}(w; \boldsymbol{\nu})\| \leq L_v \sum_{j=1}^L W_1(\mu^{(j)}, \nu^{(j)}).
\]

\item \textbf{Spatial regularity:} For each fixed $\boldsymbol{\mu}$, the map $w \mapsto v^{(\ell)}(w; \boldsymbol{\mu})$ is globally Lipschitz with constant $L_v$ on $\mathcal{K}$.
\end{enumerate}
\end{assumption}

\subsection{Mean-field convergence theorem}

\begin{theorem}[Weak convergence to mean-field limit]\label{thm:mf}
Fix $\varepsilon \in (0,1]$, $b \in \mathbb{N}$, $\delta \in (0,1)$, and $T > 0$. Under Assumptions \ref{ass:reg} and \ref{ass:velocity-reg}, let $n_\ell \to \infty$ for all $\ell$ with $n_\ell/n_j \to r_{\ell j} \in (0,\infty)$ and let $\eta \downarrow 0$ with $k\eta \to t \in [0,T]$.

Then the empirical process $\boldsymbol{\widehat\mu}_n$ converges weakly in $C([0,T], \prod_{\ell=1}^L \mathcal{P}(\mathbb{R}^{m_\ell}))$ to a unique $\boldsymbol{\mu} = (\mu^{(1)}, \ldots, \mu^{(L)})$ that solves the coupled system \eqref{eq:cont-eq} with velocity fields \eqref{eq:velocity}.

Moreover, for each $\ell$ and all $\varphi \in C_c^1(\mathbb{R}^{m_\ell})$:
\[
\frac{d}{dt}\int \varphi(w) \, d\mu^{(\ell)}(t,w) = \int \nabla\varphi(w) \cdot v^{(\ell)}(w,t) \, d\mu^{(\ell)}(t,w).
\]
\end{theorem}

\begin{proof}
The proof proceeds through four main steps: compactness, velocity field regularity verification, limit identification, and uniqueness.

\textbf{Compactness of empirical measures.}

By Assumption \ref{ass:reg}(R4), all row vectors $w_i^{(\ell)}(k)$ remain in the compact set $\mathcal{K} := \{w \in \mathbb{R}^{m_\ell} : \|w\|_\infty \leq M_\star\}$ for $t \in [0,T]$.

From Lemma \ref{lem:grad-bound}, there exists $M_{\text{grad}} < \infty$ such that:
\[
\|\nabla_{w_i^{(\ell)}}\mathcal{R}_\varepsilon(W(k))\| \leq M_{\text{grad}}.
\]

This yields uniformly bounded increments:
\[
\|w_i^{(\ell)}(k+1) - w_i^{(\ell)}(k)\| = \eta \|\nabla_{w_i^{(\ell)}}\mathcal{R}_\varepsilon(W(k))\| \leq \eta M_{\text{grad}}.
\]

For equicontinuity, given $\epsilon > 0$, choose $\delta = \epsilon/(2M_{\text{grad}})$ and $\eta < \epsilon/(6M_{\text{grad}})$. Then for $|t-s| < \delta$:
\[
W_1(\widehat\mu^{(\ell)}_{n_\ell}(t), \widehat\mu^{(\ell)}_{n_\ell}(s)) \leq \max_i \|w_i^{(\ell)}(\lfloor t/\eta \rfloor) - w_i^{(\ell)}(\lfloor s/\eta \rfloor)\| < \epsilon.
\]

By compactness of $\mathcal{P}(\mathcal{K})$ in the Wasserstein topology and the Arzelà-Ascoli theorem, $\{\boldsymbol{\widehat\mu}_n\}$ is relatively compact in $C([0,T], \prod_{\ell=1}^L \mathcal{P}(\mathcal{K}))$.

\textbf{Velocity field regularity.}

The functional derivative satisfies:
\[
\mathcal{R}^{(\ell)}_\varepsilon[\boldsymbol{\mu}](w) = \int \ell(f_{\boldsymbol{\mu},w}^{(\ell)}(x), y) \, d\pi(x,y),
\]
where $f_{\boldsymbol{\mu},w}^{(\ell)}(x)$ denotes the network output when layer $\ell$ has an additional infinitesimal mass at position $w$.

By Lemma \ref{lem:lipschitz-forward} and the chain rule, for $w, \tilde{w} \in \mathcal{K}$:
\[
|\mathcal{R}^{(\ell)}_\varepsilon[\boldsymbol{\mu}](w) - \mathcal{R}^{(\ell)}_\varepsilon[\boldsymbol{\mu}](\tilde{w})| \leq L_2 L_{\mathrm{fwd}} \|w - \tilde{w}\|_F.
\]

This establishes Lipschitz continuity of the functional derivative, ensuring that $v^{(\ell)}(w,t)$ exists almost everywhere with:
\[
\|v^{(\ell)}(w,t)\| \leq L_2 L_{\mathrm{fwd}} \quad \text{for a.e. } w \in \mathcal{K}.
\]

\textbf{Limit identification.}

Let $\boldsymbol{\mu}$ be any limit point of $\{\boldsymbol{\widehat\mu}_n\}$. For $\varphi \in C_c^1(\mathbb{R}^{m_\ell})$ and $0 \leq s < t \leq T$, we perform discrete integration by parts:

\begin{align}
&\int \varphi \, d\widehat\mu^{(\ell)}_{n_\ell}(t) - \int \varphi \, d\widehat\mu^{(\ell)}_{n_\ell}(s) \\
&= \frac{1}{n_\ell}\sum_{i=1}^{n_\ell} \sum_{k=\lfloor s/\eta \rfloor}^{\lfloor t/\eta \rfloor - 1} [\varphi(w_i^{(\ell)}(k+1)) - \varphi(w_i^{(\ell)}(k))].
\end{align}

By Taylor expansion and uniform bounds:
\[
\varphi(w_i^{(\ell)}(k+1)) - \varphi(w_i^{(\ell)}(k)) = \nabla\varphi(w_i^{(\ell)}(k)) \cdot (w_i^{(\ell)}(k+1) - w_i^{(\ell)}(k)) + O(\eta^2 M_{\text{grad}}^2).
\]

Substituting the gradient descent updates and taking limits:
\begin{align}
&\int \varphi \, d\mu^{(\ell)}(t) - \int \varphi \, d\mu^{(\ell)}(s) \\
&= -\int_s^t \int \nabla\varphi(w) \cdot v^{(\ell)}(w,r) \, d\mu^{(\ell)}(r,w) \, dr.
\end{align}

Differentiating with respect to $t$ yields the weak formulation of \eqref{eq:cont-eq}.

\textbf{Uniqueness.}

Let $\boldsymbol{\mu}, \boldsymbol{\nu}$ be two solutions with identical initial conditions. Define:
\[
d(t) := \sum_{\ell=1}^L W_1(\mu^{(\ell)}(t), \nu^{(\ell)}(t)).
\]

By Assumption \ref{ass:velocity-reg} and the contraction property of optimal transport:
\[
\frac{d}{dt} W_1(\mu^{(\ell)}(t), \nu^{(\ell)}(t)) \leq L_v d(t).
\]

Summing over $\ell$ and applying Grönwall's inequality with $d(0) = 0$ yields $d(t) = 0$ for all $t \in [0,T]$, establishing uniqueness.
\end{proof}

\subsection{Interacting particle system interpretation}

\begin{remark}[Connection to particle systems]\label{rem:particle-interpretation}
The empirical measures \eqref{eq:emp-measure} admit a natural interpretation in terms of interacting particle systems. Each row vector $w^{(\ell)}_i(k) \in \mathbb{R}^{m_\ell}$ can be viewed as the position of the $i$-th \emph{particle} in layer $\ell$ at time $k$.

Under this interpretation:
\begin{itemize}
\item The empirical measure $\widehat\mu^{(\ell)}_{n_\ell}(k) = \frac{1}{n_\ell}\sum_{i=1}^{n_\ell}\delta_{w^{(\ell)}_i(k)}$ represents the spatial distribution of particles in layer $\ell$.

\item The gradient descent update \eqref{eq:gd} becomes a system of interacting particles:
\[
w^{(\ell)}_i(k+1) = w^{(\ell)}_i(k) - \eta \nabla_{w^{(\ell)}_i} \mathcal{R}_\varepsilon(W(k)),
\]
where the force on particle $i$ depends on the positions of all particles across all layers.

\item The mean-field limit corresponds to the thermodynamic limit where the number of particles $n_\ell \to \infty$ while their individual influence vanishes as $1/n_\ell$.

\item The velocity field $v^{(\ell)}(w,t)$ in \eqref{eq:velocity} represents the drift experienced by a test particle at position $w$ in the mean-field environment.
\end{itemize}

This particle system perspective provides intuitive insight into the dynamics, while the measure-theoretic formulation in the preceding subsections provides the rigorous mathematical foundation for the analysis.
\end{remark}

\section{Gradients, chain rule, and bounds}

\subsection{Layerwise gradients with smooth quantization}

Let $W\mapsto f_W$ be defined with smooth quantizers. Then
\begin{equation}\label{eq:grad-decomp}
\nabla_{W^{(\ell)}}\mathcal{R}_\varepsilon = \mathbb{E}\Big[\partial_1\ell\big(f_W(X),Y\big)\cdot \sum_{k=\ell}^L \frac{\partial h^{(L)}}{\partial h^{(k)}}\cdot \frac{\partial h^{(k)}}{\partial W^{(\ell)}}\Big],
\end{equation}
with all Jacobians well-defined by the chain rule. The derivative $\partial \widetilde{W}^{(\ell)}_\varepsilon/\partial W^{(\ell)}$ exists and is bounded by $\varepsilon^{-1}$ entrywise; the derivative of $\beta^{(\ell)}_\varepsilon$ has entries
\[
\partial_{W^{(\ell)}_{ij}}\beta_\varepsilon^{(\ell)}(W^{(\ell)})=\frac{1}{n_\ell m_\ell}\frac{P^{(\ell)}(W^{(\ell)})_{ij}}{\sqrt{(P^{(\ell)}(W^{(\ell)})_{ij})^2+\varepsilon^2}},
\]
bounded by $(n_\ell m_\ell)^{-1}$.

\begin{lemma}[Gradient bound]\label{lem:grad-bound}
Under Assumption \ref{ass:reg}, there exist constants $C_\ell=C_\ell(\varepsilon,b,\delta)$ such that for all entries $(i,j)$,
\[
\abs{\partial_{W^{(\ell)}_{ij}}\mathcal{R}_\varepsilon}\ \le\ C_\ell\,\Big(1+\mathbb{E}\big[\abs{f_W(X)}\big]\Big),
\]
and $\nabla_{W^{(\ell)}}\mathcal{R}_\varepsilon$ is locally Lipschitz on the compact domain.
\end{lemma}

\begin{proof}
Apply \eqref{eq:grad-decomp} and bound each factor by (R2)–(R3) together with the Lipschitz constants of the smooth quantizers on the compact set of iterates (R4). The derivative of $\mathrm{sgn}_\varepsilon$ is bounded by $\varepsilon^{-1}$, the derivative of $\beta^{(\ell)}_\varepsilon$ is bounded by $(n_\ell m_\ell)^{-1}$, and $\mathrm{Quant}^{(b)}_\varepsilon$ has bounded Jacobian for fixed $\varepsilon,b,\delta$. The expectation over compactly supported $(X,Y)$ preserves these bounds, yielding the stated inequality. Local Lipschitzness follows from boundedness of second derivatives on compacta.
\end{proof}

\end{document}